\documentclass[12pt, a4paper]{article}
\usepackage{RR}
\RRNo{6422}
\usepackage{amsfonts,amssymb}
\usepackage{times}
\usepackage{a4wide}
\usepackage{latexsym}
\usepackage{url}
\usepackage[only,llbracket,rrbracket]{stmaryrd}
\usepackage[latin1]{inputenc}
\usepackage[T1]{fontenc}
\usepackage[dvips]{epsfig}

\newcommand{\iinterv}[1]{\left\llbracket#1\right\rrbracket}

\newcommand{\aplt}{\ {\raise-.5ex\hbox{$\buildrel<\over\sim$}}\ }
\usepackage{comment}

\def\vec#1{\mathchoice{\mbox{\boldmath$\displaystyle#1$}}
{\mbox{\boldmath$\textstyle#1$}}
{\mbox{\boldmath$\scriptstyle#1$}}
{\mbox{\boldmath$\scriptscriptstyle#1$}}}

\def \sur#1#2{\mathrel{\mathop{\kern 0pt#1}\limits^{#2}}}

\def\e{\mathrm{e}}

\usepackage{psfrag}

\newcounter{algorithmeligne}
\newdimen\framewidth

\def\myframe#1{%
  \framewidth=\textwidth
  \advance\framewidth by -2pt   
  \advance\framewidth by -\rightmargin
  \advance\framewidth by -\leftmargin
  \fbox{\begin{minipage}{\framewidth}#1\end{minipage}}%
}

\def\algorithme#1
{
       \setcounter{algorithmeligne}{0}
        \newcommand{\ITEM}{
                \stepcounter{algorithmeligne}
                {\footnotesize \arabic{algorithmeligne}.}
                }
        \newcommand{\lign}{\ITEM}
  \framewidth=\textwidth
  \advance\framewidth by -1pt   
  \advance\framewidth by -\rightmargin
  \advance\framewidth by -\leftmargin
\begin{center}   
  \fbox{\hspace{0.08cm}\begin{minipage}{\framewidth} 
        \small
        \,
	\vspace*{-.3cm}
        \begin{tabbing}
\hskip0.5cm\=\qquad\=\qquad\=\qquad\=\qquad\=\qquad\=\qquad\=\qquad\=\qquad\=\qquad\=\qquad\\\kill 
        #1
        \end{tabbing}
\end{minipage}\hspace{0.1cm}}%
        \end{center} 
\vspace*{-.4cm}}

  \def\mQ{{\mathbb Q}}
\def\mR{{\mathbb R}}  
\def\mZ{{\mathbb Z}}
\newcommand{\ps}[2]{\langle#1,#2\rangle}

\RRdate{Novembre 2007}
\RRetitle{Worst-Case Hermite-Korkine-Zolotarev Reduced Lattice Bases}
\RRtitle{Bases Hermite-Korkine-Zolotarev r\'eduites ``pires cas''.}
\RRauthor{Guillaume Hanrot\and 
Damien Stehl\'e\thanks{CNRS and Universit\'e de Lyon / \'{E}NS Lyon / LIP, 46 all\'ee d'Italie, 69364 Lyon Cedex 07, France.
}}

\RRabstract{
The Hermite-Korkine-Zolotarev reduction plays a central role in strong
lattice reduction algorithms. By building upon a technique introduced
by Ajtai, we show the existence of Hermite-Korkine-Zolotarev reduced
bases that are arguably least reduced. We prove that for such bases,
Kannan's algorithm solving the shortest lattice vector problem
requires~$d^{\frac{d}{2\e}(1+o(1))}$ bit operations in
dimension~$d$. This matches the best complexity upper bound known for
this algorithm. These bases also provide lower bounds on Schnorr's
constants~$\alpha_d$ and~$\beta_d$ that are essentially equal to the
best upper bounds. Finally, we also show the existence of particularly
bad bases for Schnorr's hierarchy of reductions.
}

\RRresume{ La r\'eduction d'Hermite-Korkine-Zolotarev joue un r\^ole
  central dans les algorithmes de r\'eduction forte des r\'eseaux. En
  utilisant une technique due \`a Ajtai, nous prouvons l'existence de
  bases Hermite-Korkine-Zolotarev r\'eduites qui sont les plus mal
  r\'eduites possible. Pour de telles bases, l'algorithme de Kannan
  pour la r\'esolution du probl\`eme du vecteur le plus court
  n\'ecessite $d^{\frac{d}{2\e}(1+o(1))}$ op\'erations
  \'el\'ementaires en dimension $d$, ce qui co\"{\i}ncide avec la
  meilleure borne sup\'erieure connue pour sa complexit\'e. Ces bases
  fournissent \'egalement des bornes inf\'erieures pour les constantes
  de Schnorr $\alpha_d$ et $\beta_d$, qui co\"{\i}ncident l\`a encore avec les
  meilleures bornes sup\'erieures connues. Enfin, nous montrons
  l'existence de mauvaises bases r\'eduites pour les algorithmes de la
  hi\'erarchie de Schnorr.  }

\RRmotcle{R\'eduction des r\'eseaux, probl\`eme du vecteur le plus court, 
r\'eduction HKZ, r\'eduction BKZ}
\RRkeyword{Lattice basis reduction, shortest vector problem, HKZ-reduction, 
BKZ-reduction}
\RRprojets{Cacao et Ar\'enaire}
\RRtheme{\THSym}
\URLorraine

\newtheorem{theorem}{Theorem}
\newtheorem{lemma}{Lemma}

\newtheorem{corollary}{Corollary}

\def\qed{\hfill $\Box$}

\begin{document}
\makeRR
%\maketitle

\section{Introduction}
A {\em lattice}~$L$ is a discrete subgroup of a euclidean
space~$\mR^n$. Such an object can always be written as the set of
integer linear relations of some linearly independent
vectors~$\vec{b}_1,\ldots,\vec{b}_d \in \mR^n$. The~$\vec{b}_i$'s form
a {\em basis} of~$L$. Such a representation is not unique, but all
bases share the same cardinality~$d$, called the lattice {\em
  dimension}. Another lattice invariant is the so-called lattice {\em
  volume}~$\det(L)$, which is defined as the geometric $d$-dimensional
volume of any parallelepiped~$\mathcal{P}(\vec{b}_i) = \left\{ \sum_i
y_i \vec{b}_i, y_i \in [0,1] \right\}$ spanned by a lattice
basis~$(\vec{b}_i)_i$.  When~$d \geq 2$, a given lattice has an
infinity of bases, related to one another by unimodular
transformations.  Some bases are better than others, in particular
under the light of applications such as algorithmic number
theory~\cite{Cohen95} and cryptography~\cite{NgSt01,MiGo02}. In these
applications, one is mostly interested in lattice bases made of rather
short and rather orthogonal vectors. Such bases are called {\em reduced}.
One often distinguishes between reductions that are rather weak but
can be computing efficiently and reductions that are strong but that
require a much larger amount of computational resources. The main
reduction of the first family is the celebrated
LLL-reduction~\cite{LeLeLo82}, whereas the most famous one in the
second family is the Hermite-Korkine-Zolotarev reduction (HKZ for
short). There exist compromises between LLL and HKZ reductions, such
as Schnorr's Block-Korkine-Zolotarev (BKZ) reductions~\cite{Schnorr87}
depending on a parameter~$k$: the $2$-BKZ reduction is essentially the
LLL~reduction whereas the $d$-BKZ reduction is exactly the HKZ
reduction. Other compromises have been considered
in~\cite{Schnorr87,Schnorr07,GaHoKoNg06}.

{From} the algorithmic point of view, LLL-reduction can be reached in
time polynomial in the lattice dimension. The other parameters, such
as the dimension of the embedding space and the bit-size of the
initial vectors are of small interest here since all the described
algorithms have polynomial complexities with respect to them. On the
other extreme, there are two main algorithms to compute an HKZ-reduced
basis. The first one is due to Kannan~\cite{Kannan83} and was improved
by Helfrich and Schnorr~\cite{Helfrich85,Schnorr87}. Its complexity
has been revised downwards by Hanrot and Stehl\'e~\cite{HaSt07} who
proved a~$d^{\frac{d}{2\e}(1+o(1))}$ upper bound.  The other algorithm
is due to Ajtai, Kumar and Sivakumar~\cite{AjKuSi01} and its
complexity upper bound was re-assessed recently by Nguyen and
Vidick~\cite{NgVi07}: its cost is provably bounded by~$2^{5.9 \cdot
  d}$. The latter algorithm has a much better asymptotic complexity
upper bound than Kannan's. However, it suffers from two drawbacks:
firstly, it requires an exponential space whereas Kannan's space
requirement is polynomial; secondly, it is probabilistic in the sense
that there is a tiny probability that the computed basis is not
HKZ-reduced, whereas Kannan's algorithm is deterministic. In practice,
for manageable problem sizes, it seems that adaptations of Kannan's
algorithm still outperform the algorithm of Ajtai, Kumar and
Sivakumar. One of the results of the present paper is to provide a
worst-case complexity lower bound to Kannan's algorithm which is
essentially the same as the~$d^{\frac{d}{2\e}(1+o(1))}$ complexity
upper bound: it proves that from the worst-case point of view,
Kannan's algorithm is asymptotically worse that the one of Ajtai,
Kumar and Sivakumar. In the compromises between LLL and
HKZ-reductions, an algorithm computing HKZ-reduced bases (either
Kannan's or the one of Ajtai, Kumar and Sivakumar) is used
on~$k$-dimensional bases, where~$k$ is the parameter of the
compromise. When~$k$ is greater than~$c \log d$ for some constant~$c$,
the complexities of the compromise algorithms are~$k^{O(k)}$
or~$2^{O(k)}$ depending on the chosen HKZ-reduction algorithm.
\medskip

The main result of the present paper is to prove the existence of
HKZ-reduced bases which are arguably least reduced possible. These
bases are good corner cases for strong lattice reductions. We prove
that given them as input, Kannan's algorithm costs at
least~$d^{\frac{d}{2\e}(1+o(1))}$ binary operations in dimension~$d$,
thus completing the worst-case analysis of Kannan's algorithm. This
proves that the Ajtai-Kumar-Sivakumar algorithm is strictly better
than Kannan's from the worst-case asymptotic time complexity
perspective. These lattice bases also provide lower bounds on
Schnorr's constants~$\alpha_k$ and~$\beta_k$ which play a central role
to estimate the quality of Schnorr's hierarchies of reductions. As a
by-product, we improve the best known upper bound for~$\alpha_k$, and
the lower and upper bounds essentially match. Our lower bound
on~$\beta_k$ match its best known upper bound, provided
by~\cite{GaHoKoNg06}. This gives weight to the fact that the
primal-dual reduction therein may be better than Schnorr's classical
hierarchy. Finally, we provide lattice bases that are particularly bad
for Schnorr's hierarchy of reduction algorithms.

To achieve these results, we simplify and build upon a technique
introduced by Ajtai in~\cite{Ajtai03} to show lower bounds on
Schnorr's constants~$\alpha_k$ and~$\beta_k$. These lower bounds were
of the same orders of magnitude as the best upper bounds, but with
undetermined constants in the exponents. It consists in building
random lattice bases that are HKZ-reduced with non-zero probability
and such that the quantities under investigation (e.g., Schnorr's
constants) are close to the best known upper bounds. The random
lattice bases are built from their Gram-Schmidt orthogonalisations.

\medskip

\noindent
{\sc Road-map.} In Section~\ref{se:reminder} we provide the background
that is necessary to the understanding to the rest of the article. In
Section~\ref{se:ajtai} we simplify Ajtai's method to generate lattice
bases. We use it first in Section~\ref{se:worstHKZ} to show the
existence of worst-case HKZ-reduced bases with respect to the
orthogonality of the basis vectors.  Using these bases, we provide
lower bounds to the worst-case cost of Kannan's algorithm and to
Schnorr's constants~$\alpha_k$ and~$\beta_k$, in
Section~\ref{se:kannan}. We use Ajtai's technique a second time in
Section~\ref{se:BKZ} to build lattice bases that are particularly bad
for Schnorr's hierarchy of reduction algorithms. Finally, in
Section~\ref{se:open}, we draw a list of possible natural extensions
of our work.

\medskip
\noindent
{\sc Notation.} If~$y$ is a real number, we let~$\lfloor y \rceil$
denote its closest integer (with any rule for the ambiguous cases),
and we define~$\{ y \} = y - \lfloor y \rceil$.  If~$a \leq b$, we
let~$\iinterv{a,b}$ denote the set of integers belonging to the
interval~$[a,b]$. All logarithms used are in basis~$e$. Finally,
for~$x$ a real number, we define~$(x)_+ := \max(x, 0)$.

\section{Background on Lattices}
\label{se:reminder}

We refer to~\cite{Cassels71} for a complete introduction to lattices.

\medskip
\noindent 
{\bf Gram-Schmidt orthogonalisation.}
Let~$\vec{b}_1,\ldots,\vec{b}_d$ be linearly independent vectors. We
define~$\vec{b}_i^* = \vec{b}_i - \sum_{j<i} \mu_{i,j} \vec{b}_j^*$
with~$\mu_{i,j} =
\frac{\ps{\vec{b}_i}{\vec{b}_j^*}}{\|\vec{b}_j^*\|^2}$. The~$\vec{b}_i^*$'s
are orthogonal and, for any~$i$, we have that the linear span of
the~$\vec{b}_j^*$'s for~$j \leq i$ is exactly the span of
the~$\vec{b}_j$'s for~$j\leq i$. If~$j \leq i$, we denote
by~$\vec{b}_i(j)$ the projection of~$\vec{b}_i$ orthogonally to the
vectors~$\vec{b}_1,\ldots,\vec{b}_{j-1}$. We have~$\vec{b}_i(j) =
\vec{b}_i^* + \sum_{k=j}^{i-1} \mu_{i,k} \vec{b}_k^*$.

\medskip
\noindent 
{\bf Minkowski's inequality.} For all integer $d \geq 1$, there exists
a constant~$\gamma_d$, called {\em Hermite's constant}, such that for
any~$d$-dimensional lattice~$L$ there exists a non-zero
vector~$\vec{b} \in L$ with~$\|\vec{b}\| \leq \gamma_d^{1/2} \cdot
(\det L)^{\frac{1}{d}}$. The latter relation is known as {\em
  Minkowski's inequality}. Hermite's constant satisfies~$\gamma_d \leq
d$.  Asymptotically, one has~$\frac{1.744 d}{2\pi \e}(1+o(1)) \geq
\gamma_d \geq \frac{d}{2\pi \e}(1+o(1))$ (see~\cite{KaLe78} for the upper
bound).  We define the {\em minimum} of a lattice~$L$ as the length of
a shortest non-zero vector, and we let it be denoted by~$\lambda(L)$.
Minkowski's inequality can be easily restated in terms of the
Gram-Schmidt orthogonalisation of any basis~$(\vec{b}_i)_i$ of~$L$
since~$\det (L) = \prod_i \|\vec{b}_i^*\|$:
\[
\lambda(L) \leq \sqrt{d} \cdot \left(\prod_{i=1}^d
\|b_i^*\|\right)^{\frac{1}{d}}.
\]

\medskip
\noindent
{\bf Hermite-Korkine-Zolotarev reduction.} A basis~$(\vec{b}_i)_i$ of
a lattice~$L$ is said to be {\em HKZ-reduced} if its first vector reaches
the minimum of~$L$ and if orthogonally to~$\vec{b}_1$ the
other~$\vec{b}_i$'s are themselves HKZ-reduced. This implies that for
any~$i$ we have~$\|\vec{b}_i^*\| \leq \sqrt{d-i+1} \cdot
\left(\prod_{j=i}^d \|\vec{b}_j^*\| \right)^{\frac{1}{d-i+1}}$. We
call these~$d-1$ inequalities the {\em primary Minkowski
  inequalities}. Many other Minkowski-type inequalities are satisfied
by an HKZ-reduced basis since the HKZ-reducedness
of~$(\vec{b}_1,\ldots,\vec{b}_d)$ implies the HKZ-reducedness of any
basis~$(\vec{b}_i(i),\ldots,\vec{b}_j(i))$ for any~$i \leq j$.

\medskip
\noindent
{\bf Schnorr's hierarchies of reductions.}  A
basis~$(\vec{b}_1,\ldots,\vec{b}_d)$ is called {\em
  Block-Korkine-Zolotarev reduced} with block-size~$k$ ($k$-BKZ
for short) if for any~$i \leq d-k+1$ the~$k$-dimensional
basis~$(\vec{b}_i(i),\ldots,\vec{b}_{i+k-1}(i))$ is HKZ-reduced. This
reduction was initially called~$k$-reduction
in~\cite{Schnorr87}. Schnorr also introduced the~block-$2k$-reduction:
a basis~$(\vec{b}_1,\ldots,\vec{b}_d)$ is block-$2k$-reduced if for
any~$i \leq \lceil d/k \rceil -2$, the
basis~$(\vec{b}_{ik+1}(ik+1),\ldots,\vec{b}_{j}(ik+1))$ with~$j=
\min(d,(i+2)k)$ is HKZ-reduced. Any~$2k$-BKZ-reduced basis is
block-$2k$-reduced and any block-$2k$-reduced basis
is~$k$-BKZ-reduced. In the following, we will concentrate on the BKZ
hierarchy of reductions.

\medskip
\noindent
{\bf Schnorr's constants.} In order to analyze the quality of
the~$k$-BKZ and block-$2k$ reductions, Schnorr introduced the
constants
\[ 
\alpha_k =  
\max_{(\vec{b}_i)_{i \leq k} \mbox{\scriptsize HKZ-reduced}}
\ \frac{\|\vec{b}_1\|^2}{\|\vec{b}_k^*\|^2} \ \ \mbox{ and } \ 
\beta_k =  
\max_{(\vec{b}_i)_{i \leq 2k} \mbox{\scriptsize HKZ-reduced}}
\ \left( 
\frac{\prod_{i \leq k} \|\vec{b}_i^*\|^2}{\prod_{i>k} \|\vec{b}_i^*\|^2}
\right)^{\frac{1}{k}}.
\]
The best known upper bounds on~$\alpha_k$ and~$\beta_k$ are~$k^{1+\log
  k}$ and~$\frac{1}{10}k^{2\log 2}$ (see~\cite{Schnorr87,GaHoKoNg06}).
We will improve the upper bound on~$\alpha_k$ in
Section~\ref{se:kannan}.  Any~$k$-BKZ-reduced
basis~$(\vec{b}_1,\ldots,\vec{b}_d)$ of a lattice~$L$
satisfies~$\|\vec{b}_1\| \leq \min \left( k^{\frac{d-1}{k-1}},
\alpha_k^{\frac{d-1}{k-1}-1}\right) \lambda (L)$.
Ajtai~\cite{Ajtai03} showed that~$\alpha_k \geq k^{c \log k}$ for some
constant~$c$, so that the first upper bound is stronger than the
second one.  Furthermore, every block-$2k$-reduced
basis~$(\vec{b}_1,\ldots,\vec{b}_{mk})$ of a lattice~$L$
satisfies~$\|\vec{b}_1\| \leq \sqrt{k} \sqrt{\beta_k}^{m-1}
\lambda(L)$ (see~\cite{Schnorr87,Schnorr94}).

\section{Ajtai's Drawing of HKZ-Reduced Bases}
\label{se:ajtai}

Consider a dimension~$d>0$ and a function $f :
\iinterv{1,d} \rightarrow \mR^+ \setminus \{ 0\}$.  By generalising an
argument due to Ajtai~\cite{Ajtai03}, we prove that one can build a
$d$-dimensional lattice basis which is HKZ-reduced and such
that~$\|\vec{b}_i^*\| = f(i)$, under a ``Minkowski-type'' condition
for the values of~$f$.

\begin{theorem}
\label{th:ajtai}
Let~$d>0$ and $f: \iinterv{1, d} \rightarrow \mR^+ \setminus \{0\}$. 
Assume that for any~$j \leq d$, one has
\[
\sum_{i=1}^{j-1} \left(\frac{2\pi \e}{j-i}\right)^{\frac{j-i}{2}} 
\left(1 - \left(\frac{f(j)}{f(i)}\right)^2\right)_+^{\frac{j-i}{2}} 
\left( \prod_{k=i}^j \frac{f(i)}{f(k)} \right)
< 1.
\]
Then there exists an HKZ-reduced basis $(\vec{b}_1,\ldots,\vec{b}_d)$
with~$\|\vec{b}_i^*\| = f(i)$.
\end{theorem}

The condition above might seem intricate at first glance, though it is in
fact fairly natural. The term~$(j - i)^{-\frac{j-i}{2}} \prod_{k=i}^j
\frac{f(i)}{f(k)}$ resembles Minkowski's inequality. It is
natural that it should occur for all~$(i, j)$, since for an
HKZ-reduced basis Minkowski's inequality is satisfied for all
bases~$(\vec{b}_i(i), \ldots, \vec{b}_j(i))$. Another way of
stating this is that a necessary condition for a basis to be
HKZ-reduced would be
\[
\forall j \leq d, \ \sum_{i=1}^{j-1} \left(
4\gamma_{j-i+1} \right)^{-\frac{j-i}{2}}
\left( 1 - \left(\frac{f(j)}{f(i)}\right)^2\right)^{\frac{j-i}{2}}  
\left( \prod_{k=i}^j \frac{f(i)}{f(k)} \right) < 1. 
\]
This is merely a restatement of the fact that, since Minkowski's
inequality is verified for any pair~$(i, j)$, the $i$-th term is at
most $2^{-(j-i)}$, so that the sum is $<1$. In view of the fact that
asymptotically~$\gamma_d \leq \frac{1.744 d}{2 \pi \e}(1+o(1))$, we
see that we are not far from an optimal condition.

Lemma~\ref{le:core} is the core of the proof of
Theorem~\ref{th:ajtai}. It bounds the probability that when a random
basis~$(\vec{b}_1, \ldots, \vec{b}_d)$ is built appropriately, any
lattice vector~$\sum_i x_i \vec{b}_i$ with~$x_d \neq 0$ will be longer
than~$\vec{b}_1$.

\begin{lemma} 
\label{le:core}
Let~$(\vec{b}_1, \ldots, \vec{b}_{d-1})$ be a lattice basis and
let~$\vec{b}_d$ be a random vector. We suppose that:
\begin{enumerate}
\item For any~$i\leq d$, we have~$\|\vec{b}_i^*\| = f(i)$.
\item The~$\mu_{d,i}$'s for~$i<d$ are independent random variables uniformly
distributed in~$[-1/2,1/2]$.
\end{enumerate}
Let~$p$ be the probability that there exists~$(x_1,\ldots,x_d)$
with~$x_d \neq 0$ such that~$\left\|\sum_i x_i \vec{b}_i\right\| \leq
\|\vec{b}_1\|$. Then:
\[
p \leq \left(\frac{2\pi \e}{d-1}\right)^{\frac{d-1}{2}} \sum_{x > 0}
\left( 1 - \left(\frac{x f(d)}{f(1)}\right)^2 \right)_+^{\frac{d-1}{2}}
\left( \prod_{i<d} \frac{f(1)}{f(i)} \right).
\]
\end{lemma}

\begin{proof}
Wlog we can assume $x_d > 0$. We can write
$$\sum_{i\leq d} x_i
\vec{b}_i = \sum_{i\leq d} \left( x_i + \sum_{j=i+1}^d \mu_{j,i} x_j
\right) \vec{b}_i^*.$$  
For~$i \leq d$, we define~$u_i = x_i +
\left\lfloor \sum_{j=i+1}^d \mu_{j,i} x_j\right \rceil$ and~$\delta_i
= \left\{ \sum_{j=i+1}^d \mu_{j,i} x_j \right\}$. Notice
that~$\delta_i = \left\{\mu_{d,i} x_d + \sum_{j=i+1}^{d-1} \mu_{j,i}
x_j \right\}$ is made of a random term ($\mu_{d,i} x_d$) and a
constant term ($\sum_{j=i+1}^{d-1} \mu_{j,i} x_j$). Since~$x_d \neq 0$
and since the~$\mu_{d,i}$'s are distributed independently and
uniformly in~$[-1/2,1/2]$, the same holds for the~$\delta_i$'s (for each fixed choice of~$(x_1,\ldots,x_d)$). The event defining~$p$ can
thus be rewritten as
\[
\exists u_d \in \mZ_{>0}, \  
\exists (u_1, \dots, u_{d-1}) \in \mZ^{d-1}, \
\sum_{i<d} (u_i + \delta_i)^2 f(i)^2 \leq f(1)^2 - u_d^2 f(d)^2.
\]

The probability of this event is~$0$ if~$f(1)^2 - u_d^2 f(d)^2 < 0$.
We shall thus assume in the sequel that~$0 < u_d \leq f(1)/f(d)$. The
probability~$p$ is then bounded by
\[
\sum_{u_d \in \mZ \setminus \{ 0 \}} \ \sum_{(u_1, \dots, u_{d-1}) \in \mZ^{d-1}} 
\Pr \left(\sum_{i<d} (u_i+ \delta_i)^2 f(i)^2 
\leq f(1)^2 - u_d^2 f(d)^2\right).
\]

Let~$c>0$ be an arbitrary constant. We can
estimate the last upper bound by using the inequality
\[
\Pr \left(\sum_{i<d} (u_i+ \delta_i)^2 f(i)^2 
\leq f(1)^2 - u_d^2 f(d)^2\right)
\leq \int_{\vec{\delta} \in \left[-\frac{1}{2}, \frac{1}{2}\right]^{d-1}} 
\exp \left(c - c\frac{\sum_{i<d} (u_i + \delta_i)^2 f(i)^2}
{f(1)^2 - u_d^2f(d)^2}\right) d\vec{\delta}.
\]

Summing over the~$u_i$'s, we obtain the estimate
\begin{eqnarray*}
\sum_{\vec{u} \in \mZ^{d-1}}
\int_{\vec{\delta} \in \left[-\frac{1}{2}, \frac{1}{2}\right]^{d-1}} 
\exp \left(c - c\frac{\sum_{i<d} (u_i + \delta_i)^2 f(i)^2}
{f(1)^2 - u_d^2f(d)^2}\right) d\vec{\delta}
&= &  \\
&&\hspace*{-2.5cm}
\int_{\mR^{d-1}} 
\exp\left(c - c\frac{\sum_{i<d} \delta_i^2 f(i)^2}{f(1)^2 - u_d^2 f(d)^2}\right)
d\vec{\delta} \\
& &  \hspace*{-2.5cm} = 
\e^c \prod_{i<d} \int_{\mR} 
\exp\left(- c\frac{\delta_i^2 f(i)^2}{f(1)^2 - u_d^2 f(d)^2}\right)
d\delta_i
\\
& & \hspace*{-2.5cm} = 
\e^c \left(\frac{\pi}{c}\right)^{\frac{d-1}{2}} 
\left(1 - \left(\frac{u_d f(d)}{f(1)}\right)^2\right)^{\frac{d-1}{2}} 
\prod_{i<d} \frac{f(1)}{f(i)}.
\end{eqnarray*}
Taking~$c = (d-1)/2$ and summing over $x_d = u_d > 0$ yields the bound
that we claimed. Recall that the terms corresponding to~$u_d >
f(1)/f(d)$ do not contribute. \qed
\end{proof}

We now proceed to prove Theorem~\ref{th:ajtai}.  We build the basis
iteratively, starting with~$\vec{b}_1$, chosen arbitrarily
with~$\|\vec{b}_1\| = f(1)$. Assume now that~$\vec{b}_1, \ldots,
\vec{b}_{j-1}$ have already been chosen with~$\|\vec{b}_i^*\|=f(i)$
for~$i<j$ and that they are HKZ-reduced. We choose~$\vec{b}_j$ as
$\vec{b}_j^* + \sum_{k<j} \mu_{j,k} \vec{b}_k^*$ such
that~$\|\vec{b}_j^*\| = f(j)$ and the random variables~$\left(
\mu_{j,k} \right)_{k<j}$ are chosen uniformly and independently
in~$[-1/2, 1/2]$. Let~$p_{i,j}$ be the probability that the
vector~$\vec{b}_i^*$ is not a shortest non-zero vector
of~$L(\vec{b}_i(i),\ldots,\vec{b}_j(i))$.  This means that there exist
integers~$(x_i,\ldots,x_j)$ such that
$$\left\|\sum_{k=i}^j x_k \vec{b}_k(i)\right\| < \|\vec{b}_i^*\|.$$ 
Since~$(\vec{b}_1, \ldots,
\vec{b}_{j-1})$ is HKZ-reduced, so
is~$(\vec{b}_i(i),\ldots,\vec{b}_j(i))$ and thus we must have~$x_j
\neq 0$. Lemma~\ref{le:core} gives us
\begin{eqnarray*}
p_{i,j} 
& \leq  &
\left(\frac{2\pi \e}{j-i}\right)^{\frac{j-i}{2}} \sum_{x > 0}
\left( 1 - \left(\frac{x f(j)}{f(i)}\right)^2 \right)_+^{\frac{j-i}{2}}
\left( \prod_{k=i}^{j-1} \frac{f(i)}{f(k)} \right) \\
& \leq &
\left(\frac{2\pi \e}{j-i}\right)^{\frac{j-i}{2}} 
\left(\frac{f(i)}{f(j)}\right) 
\left( 1 - \left(\frac{f(j)}{f(i)}\right)^2 \right)_+^{\frac{j-i}{2}}
\left( \prod_{k=i}^{j-1} \frac{f(i)}{f(k)}\right) \\
& \leq &
\left(\frac{2\pi \e}{j-i}\right)^{\frac{j-i}{2}} 
\left( 1 - \left(\frac{f(j)}{f(i)}\right)^2 \right)_+^{\frac{j-i}{2}}
\left( \prod_{k=i}^j \frac{f(i)}{f(k)} \right). 
\end{eqnarray*}
We conclude the proof by observing that the probability of
non-HKZ-reducedness of~$(\vec{b}_1,\ldots,\vec{b}_j)$ is at
most~$\sum_{i < j} p_{i,j}$. By hypothesis, this quantity
is~$<1$. Overall, this means that there exist~$\mu_{i,j}$'s such that
$(\vec{b}_1,\ldots,\vec{b}_j)$ is HKZ-reduced.  
\qed

The proof of the lemma and the derivation of the theorem may not seem
tight. For instance, summing over all possible~$(u_1, \ldots, u_d)$
might seem pessimistic in the proof of the lemma. We do not know how
to improve the argument apart from the~$x_d$ part, for which, when~$j-i$ 
is large, the term
\[
\sum_{x>0} \left( 1 - \left(x \frac{f(j)}{f(i)}\right)^2
\right)_+^{\frac{j-i}{2}}
\]
could be interpreted as a Riemann sum corresponding to the integral
\[
\frac{f(i)}{f(j)} \cdot 
\int_{0}^{\pi/2} \sin^{j-i+1} x\, \mathrm{d}x 
\approx \frac{f(i)}{f(j)} \cdot \sqrt{\frac{\pi}{2(j - i + 1)}}.
\] 

Notice however that if one uses the same technique to look for vectors
of lengths smaller than~$\sqrt{c \cdot d} \cdot
\left(\prod_{i<d}f(i)\right)^{\frac{1}{d}}$ instead of~$f(1)$, one
finds that there exists a lattice where there is no vector shorter
than this length (with $x_d \neq 0$) as soon as~$c < \frac{1}{2\pi
\e}$. We thus recover, up to the restriction~$x_d \neq 0$, the
asymptotic lower bound on Hermite's constant. As a consequence, it
seems that the main hope of improvement would be to replace the sum
(in the proof of the theorem) by a maximum, or something intermediate.
Replacing by a maximum seems quite difficult. It would require to
prove that, if vectors of lengths~$\leq \| \vec{b}_1\|$ exist, then
one of them has~$x_d \neq 0$, at least almost surely. A deeper
understanding of that kind of phenomenon would allow one to obtain
refined versions of Theorem~\ref{th:ajtai}.

\begin{comment}
\begin{remark}
Note that by enforcing $\mu_{d,d-1}$ to be~$1/2$, we can improve the
criterion. It then suffices to have
\[
\sum_{i=1}^{j-1} \left(\frac{2\pi \e}{j-i}\right)^{\frac{j-i}{2}} 
\left( 1 - \frac{f(j)^2 + f(j-1)^2/4}{f(i)^2} \right)_+^{\frac{j-i}{2}} 
\prod_{k=i}^{j} \frac{f(i)}{f(k)},
\]
where $(x)_+ = \max(x, 0)$. When $f$ decreases very slowly (so that the last
product is close to 1) this allows one to discard many terms in the sum. 
\end{remark}
\end{comment}

\section{Worst-Case HKZ-reduced Bases}
\label{se:worstHKZ}

This section is devoted to the construction of an explicit
function~$f$ satisfying the conditions of Theorem~\ref{th:ajtai} as
tightly as possible. In order to make explicit the fact that~$f$
depends on the underlying dimension~$d$, we shall write~$f_d$ instead
of~$f$. Note that though~$f(i)$ will depend on~$d$, this will not be
the case for~$f(d - i)$. Suppose that the basis~$(\vec{b}_i)_i$ is
HKZ-reduced. Then~$f_d$ must satisfy Minkowski-type inequalities,
namely:
\[
\forall i<j, \ f_d(i) \leq \sqrt{\gamma_{j-i+1}} \cdot \left(
\prod_{k=i}^j f_d(k) \right)^{\frac{1}{j-i+1} }.
\]

We choose~$f_d$ according to the strongest of those conditions, namely
those we called the primary Minkowski inequalities, i.e., with~$j =
d$. It is known (see~\cite{PeZw07} for example) that this set of
conditions does not suffice for an HKZ-reduced basis to exist.  We
thus expect to have to relax somehow these constraints. We will also
replace the Hermite constant (known only for~$d \leq 8$ and~$d = 24$)
by a more explicit term.  For these reasons, we introduce
\[
f_{\psi, d}(i) = \sqrt{\psi(d-i+1)} \cdot \left( \prod_{k=i}^d f_{\psi, d}(k)
\right)^{\frac{1}{d-i+1}},
\] 
where~$\psi$ is be chosen in the sequel.  This equation
uniquely defines~$f_{\psi, d}(i)$ for all~$i$ once we set~$f_{\psi, d}(d) =
1$. 

\begin{theorem}
\label{th:main}
Let $\psi(x) = C \cdot x$ with~$C=\exp(-6)$. Then, for all~$1 \leq i <
j \leq d$, we have
\[
({j - i + 1})^{-\frac{j - i}{2}} 
\left(1 - \left(\frac{f_{\psi, d}(j)}{f_{\psi, d}(i)}\right)^2
\right)_+^{\frac{j-i}{2}} 
\left( \prod_{k=i}^j \frac{f_{\psi, d}(i)}{f_{\psi, d}(k)} \right) 
\leq  
\left(2 \pi \e (\sqrt{\e} + 1)^2 \right)^{-\frac{j-i}{2}}.
\]
\end{theorem}

Thanks to Theorem~\ref{th:ajtai}, we obtain the following.

\begin{corollary}
\label{co:shape}
Let~$\psi$ be as in the previous theorem. There exist HKZ-reduced
bases with
\[
\|\vec{b}_i^*\| = f_{\psi, d}(i) = \sqrt{d - i + 1} \cdot \prod_{l=i+1}^d 
\left(C(d - l + 2)\right)^{\frac{1}{2(d - l + 1)}}.
\]
Moreover, when~$d-i$ grows to infinity, we have
\[
\|\vec{b}_i^*\| = (d - i + 1)^{\frac{1+\log C}{2}} 
\cdot \exp\left(\frac{\log^2 (d - i + 1)}{4}+O(1)\right).
\]
\end{corollary}

The proof of the Theorem~\ref{th:main} follows from elementary
analytical considerations. The elementary and somewhat technical
nature of this proof leads us to postpone it to an appendix. It can be
skipped without inconvenience for the general progression of the
paper. We only give here an overview of the strategy. 

First, we prove that~$(j - i + 1)^{-\frac{j-i}{2}} \left(
\prod_{k=i}^j \frac{f(i)}{f(k)} \right) < 1$. Then, in order to prove
that the whole term is actually smaller than~$\left(2\pi \e (\sqrt{\e}
+ 1)^2\right)^{\frac{i-j}{2}}$, we need to consider four different
cases. Let us write~$a = d - i + 1$ and~$b = d - j + 1$.  This change
of variables makes the problem independent of~$d$.
 
\begin{itemize}
\item When~$a$ and~$b$ are very close, i.e., $a \geq b \geq a - 1.65
  \frac{a}{(\log a)^3}$, the term $(1 - (f(j)/f(i))^2)$ can be made
  arbitrarily small when~$a$ grows to infinity.  For~$a$ large enough,
  this yields a sufficiently small exponential term.
\item When~$a$ and~$b$ are not too close but not too far either, i.e.,
  $a - 1.65\frac{a}{(\log a)^3} \geq b \geq \kappa a$ for any
  constant~$\kappa$, the term~$(j - i + 1)^{-\frac{j-i}{2}} \left(
  \prod_{k=i}^j \frac{f(i)}{f(k)} \right)$ is decreasing
  exponentially, at a rate which can be made arbitrarily large for $a$
  large enough (thanks to the ``$x$'' part of $\psi(x)$).
\item When~$a / b \rightarrow +\infty$, the ``$C$'' part of~$\psi(x)$
  provides an exponential term.
\item Finally, for small~$a$ (the arguments used in the previous zones
  only work when~$a$ is large enough), we have to perform numerical
  computations to check that the inequality is indeed true.
\end{itemize}

\noindent
{\em Proof of the corollary.}  According to Theorem~\ref{th:main}, we
have
\begin{eqnarray*}
\sum_{i=1}^{j-1} \left(\frac{2\pi \e}{j-i}\right)^{\frac{j-i}{2}} 
\left( 1 - \left(\frac{f(j)}{f(i)}\right)^2\right)^{\frac{j-i}{2}} 
\left( \prod_{k=i}^j \frac{f(i)}{f(k)} \right) 
& \leq & 
\sum_{i=1}^{j-1} \left(\frac{j-i+1}{j-i}\right)^{\frac{j-i}{2}} 
\left(\sqrt{\e} + 1\right)^{-(j-i)}\\ 
& < & 
\sqrt{\e} \cdot \sum_{i\geq 1} (\sqrt{\e} + 1)^{-i} = 1. 
\end{eqnarray*}
The first part of the result follows from Theorem~\ref{th:ajtai}
and basic computations that are actually detailed in the appendix
(Lemma~\ref{le:explf}). 
For the second part, note that our choice of~$\psi$ gives
\[
2 \log f_{\psi, d}(i) = \log (d - i + 1) + \sum_{l = i+1}^d
\frac{\log C + \log(d - l + 2)}{d - l + 1}.
\]
Suppose that~$d-i \rightarrow + \infty$. We have
\begin{eqnarray*}
\left| \sum_{l = i+1}^d \frac{\log(d - l + 2)}{d - l + 1} - \int_{i}^{d} 
\frac{\log(d - x + 1)}{d - x + 1}\,\mathrm{d}x \right| & \leq & 
\left| \sum_{l = i+1}^d \frac{\log(d - l + 1)}{d - l + 1} 
- \int_{i}^{d}
\frac{\log(d - x + 1)}{d - x + 1} \, \mathrm{d}x
 \right|\\
&&\hspace*{3cm} + \sum_{l = i+1}^d \frac{1}{(d-l+1)^2}.\\
&& \hspace*{-3cm} \leq 
O(1) + \sum_{l=i+1}^d \int_{l-1}^{l} \left| \frac{\log(d - l + 1)}
{d - l + 1} 
- \frac{\log(d - x + 1)}{d - x + 1} \right| \,\mathrm{d}x \\
&& \hspace*{-3cm} \leq O(1) + 
\sum_{l=i+1}^d \max_{x\in[l-1,l]} \frac{|1 - \log(d - x + 1)|}{(d - x + 1)^2} 
\ = \ O(1).
\end{eqnarray*}

Classically, we also have
\[
\left| \sum_{l = i+1}^d \frac{\log C}{d - l + 1} - \log (C) \cdot 
\log(d - i + 1) \right| =  O(1).
\]
The result follows from the fact that~$\int_i^d \frac{\log(d - x +
  1)}{d - x + 1} \,\mathrm{d}x = \frac{\log^2 (d-i+1)}{2}$. \qed

As a direct consequence of the Corollary, we also have
\begin{corollary}
Let~$\psi$ be as in the previous theorem. There exist dual-HKZ-reduced
bases with
\[
\|\vec{b}_i^*\| = f_{\psi, d}(i) = (\sqrt{d - i + 1})^{-1} \cdot \prod_{l=i+1}^d 
\left(C(d - l + 2)\right)^{-\frac{1}{2(d - l + 1)}}.
\]
Moreover, when~$d-i$ grows to infinity, we have
\[
\|\vec{b}_i^*\| = (d - i + 1)^{-\frac{1+\log C}{2}} 
\cdot \exp\left(-\frac{\log^2 (d - i + 1)}{4}+O(1)\right).
\]
\end{corollary}

\section{Lower Bounds Related to the HKZ-Reduction}
\label{se:kannan}

The HKZ-reduced bases that we built in the previous section provide
lower bounds to several quantities. It gives a lower bound on the
complexity of Kannan's algorithm for computing a shortest non-zero
vector~\cite{Kannan83} that matches the best known upper
bound~\cite{HaSt07}. It also provides essentially optimal lower bounds
to Schnorr's constants~$\alpha_k$ and~$\beta_k$.

\subsection{Reminders on Kannan's Algorithm}

A detailed description of Kannan's algorithm can be found
in~\cite{Schnorr87}. Its aim is to HKZ-reduce a given
basis~$(\vec{b}_1,\ldots,\vec{b}_d)$. To do this, it first
quasi-HKZ-reduces it, which means that~$\|\vec{b}_1\| \leq 2
\|\vec{b}_2^*\|$ and the basis~$(\vec{b}_2(2),\ldots,\vec{b}_d(2))$ is
HKZ-reduced. After this first step, it finds all
solutions~$(x_1,\ldots,x_d)\in\mZ^d$ to the equation
\begin{equation}
\label{eq:enumerate}
\left\|\sum_{i=1}^d x_i \vec{b}_i \right\| \leq \|\vec{b}_1\|.
\end{equation}
It keeps the shortest non-zero vector~$\sum_{i=1}^d x_i \vec{b}_i$, which
attains the lattice minimum, extends it into a lattice basis and
HKZ-reduces the projection of the last~$d-1$ vectors orthogonally to
the first one.

The computationally dominant step is the second one, i.e., solving
Equation~(\ref{eq:enumerate}). It is performed by enumerating all
integer points within hyper-ellipsoids. Equation~(\ref{eq:enumerate})
implies that:
\[ 
|x_d| \cdot \|\vec{b}_d^*\| \leq \|\vec{b}_1\|.
\]
We consider all the possible integers~$x_d$ that satisfy this
equation. For any of them, we consider the following equation, which
also follows from Equation~(\ref{eq:enumerate}):
\[ 
|x_{d-1} + \mu_{d,d-1} x_d| \cdot \|\vec{b}_{d-1}^*\| \leq 
\left(\|\vec{b}_1\|^2 - x_d \|\vec{b}_d^*\|^2\right)^{1/2}.
\] 
This gives a finite number of possibilities for the integer~$x_{d-1}$
to be explored. 

Suppose that~$(x_{i+1},\ldots,x_d)$ have been
chosen. We then consider the following consequence of
Equation~(\ref{eq:enumerate}):
\[
\left| x_i + \sum_{j>i} \mu_{j,i} x_j \right| \cdot \|\vec{b}_i^*\| \leq
\left(\|\vec{b}_1\|^2 - \sum_{j>i} \left(x_j+\sum_{k>j} \mu_{k,j} x_k
\right) \|\vec{b}_j^*\|^2 \right)^{1/2},
\]
which gives a finite number of possibilities to be considered for the
integer~$x_i$.

Overall, Equation~(\ref{eq:enumerate}) is solved by enumerating all
the integer points within the hyper-ellipsoids~$\mathcal{E}_i =
\left\{ (y_i,\ldots,y_d) \in \mR^{d-i+1}, \|\sum_{j>i} y_j
\vec{b}_j(i)\| \leq \|\vec{b}_1\| \right\}$.

\subsection{On the cost of Kannan's algorithm}

In this subsection, we provide a worst-case complexity lower bound to
Kannan's algorithm by considering that the worst-case HKZ-reduced
bases built in he previous section. For these, the first step of
Kannan's algorithm has no effect, and we give a lower-bound to the
cost of the second one by providing a lower bound to the sum of the
cardinalities of the sets~$\mathcal{E}_i \cap \mZ^{d-i+1}$.

\begin{lemma}
\label{le:kannan_lower}
Let~$(\vec{b}_1,\ldots,\vec{b}_d)$ be a lattice basis. The number of
points enumerated by Kannan's algorithm is at least the sum of the
number of integer points in each of the hyperellipsoids
\[
{\mathcal E}'_i = \left\{ (y_i,\ldots,y_d) \in (\mR \setminus
\{0\})^{d-i+1}, \sum_{j\geq i} y_j^2 \|\vec{b}_j^*\|^2 \leq \frac{4}{5}
\|\vec{b}_1\|^2 \right\}.
\]
\end{lemma}

\begin{proof}
Let $\phi: \mR^{d-i+1} \rightarrow \mR^{d-i+1}$ be defined by
$\phi(y_i,\ldots,y_d) = (z_i,\ldots,z_d)$ such that~$z_i = y_i -
\left\lfloor \sum_{k>j} \mu_{k,j}z_j \right\rceil$. The function~$\phi$
is injective. Indeed, $\phi(y_i,\ldots,y_d) = (z_i,\ldots,z_d)$
implies that $y_j = z_j + \left\lfloor \sum_{k>j} \mu_{k,j}z_j
\right\rceil$, which means that~$(z_i,\ldots,z_d)$ uniquely
determines~$(y_i,\ldots,y_d)$.  Furthermore,
\[
\sum_{j \geq i} z_j \vec{b}_j(i) =
\sum_{j \geq i} \left(z_j + \sum_{k>j} \mu_{k,j} z_k\right) \vec{b}_j^*=
\sum_{j \geq i} (y_j + \delta_j) \vec{b}_j^*,
\]
for some~$\delta_j \in [-1/2, 1/2]$. Hence, for~$(y_i,\ldots, y_d) \in
{\mathcal E}'_i \cap \mZ^{d-i+1}$, the~$z_i$'s are integers and
\[
\left\|\sum_{j \geq i} z_j \vec{b}_j(i) \right\| 
= \sum_{j \geq i} (y_j + \delta_j)^2 \|\vec{b}_j^*\|^2
\leq \sum_{j \geq i} \frac{5}{4} y_j^2 \|\vec{b}_j^*\|^2
\leq \|\vec{b}_1\|^2.
\]
This implies that if~$(y_i,\ldots,y_d) \in {\mathcal E}'_i \cap
\mZ^{d-i+1}$ then~$\phi(y_i,\ldots,y_d) \in {\mathcal
  E}_i \cap \mZ^{d-i+1}$ is indeed considered.
\qed
\end{proof}

We can now provide a lower bound to the cost of Kannan's algorithm.
This lower bound is essentially the best possible, since it matches
the upper bound of~\cite{HaSt07}. This also shows that the worst-case
HKZ-reduced bases are worst-case inputs for Kannan's algorithm.

\begin{theorem}
\label{th:kannan_lower}
Let~$(\vec{b}_1,\ldots,\vec{b}_d)$ be a lattice basis.  Let~$i$ be 
such that~$\|\vec{b}_j^* \| \leq
\frac{\|\vec{b}_1\|}{\sqrt{d}}$ for all~$j \geq i$. Then, the number
of points considered by Kannan's algorithm is at least
\[
2^{-d+i-1} \prod_{j \geq i} \frac{\|\vec{b}_1\|}{\sqrt{d} \|\vec{b}_j^*\|}.
\]
In particular, given as input the basis built in the previous section,
Kannan's algorithm performs at least~$d^{\frac{d}{2\e}(1+o(1))}$
operations.
\end{theorem}

\begin{proof}
The set~${\mathcal E}'_i$ contains the subset 
\[
\prod_{j \geq i}^d \left(\left[-\frac{\|\vec{b}_1\|}{\sqrt{d}
    \|\vec{b}_j^*\|}, \frac{\|\vec{b}_1\|}{\sqrt{d} \|\vec{b}_j^*\|}
  \right] \setminus \{ 0\} \right) .
\]
This means that the cardinality of~${\mathcal E}'_i \cap \mZ^{d-i+1}$
is greater than
\[
\prod_{j \geq i} \left( 2 \left \lfloor \frac{\|\vec{b}_1\|}{\sqrt{d}
  \|\vec{b}_j^*\|}\right \rfloor - 1 \right) 
\geq 
\prod_{j \geq i}
\left(2 \frac{\|\vec{b}_1\|}{\sqrt{d} \|\vec{b}_j^*\|} - \frac{3}{2} \right) 
\geq
\frac{1}{2^{d-i+1}} \prod_{j \geq i}^d 
\frac{\|\vec{b}_1\|}{\sqrt{d} \|\vec{b}_j^*\|}.\]
This proves the first part of the theorem. It remains to evaluate this
quantity for the basis built in the previous section. For this basis,
we have, for any~$i \leq d$,
\[
\prod_{j \geq i} \frac{\|\vec{b}_i^*\|}{\|\vec{b}_j^*\|} = 
(\sqrt{C(d - i + 1)})^{d-i+1}. 
\]

As a consequence, the
number of operations performed by Kannan's algorithm given this basis
as input is greater than
\[
\left( \frac{C (d-i+1)}{4d} \right)^{\frac{d-i+1}{2}}
\cdot \left( \frac{\|\vec{b}_1\|}{\|\vec{b}_i^*\|} \right)^{d-i+1},
\]
for any~$i$ such that~$\|\vec{b}_j^* \| \leq
\frac{\|\vec{b}_1\|}{\sqrt{d}}$ for $j \geq i$.  We choose~$i = \left
\lfloor d\left(1 - \frac{1}{\e}\right) + \alpha\frac{d}{\log d} \right
\rceil$, for some $\alpha$ to be fixed later. Let~$j \geq
i$. According to Corollary~\ref{co:shape}, if~$d-j \rightarrow +
\infty$, we have
\begin{eqnarray*}
2 \log \frac{\|\vec{b}_j^*\|}{\|\vec{b}_1\|} 
& = & 
\frac{\log^2(d - j + 1) - \log^2 d}{2} + 
(1+\log C) \left(\log (d-j+1) - \log d \right)
+ O(1)\\
& \leq & 
\log \frac{d-j+1}{d} \left(\log d + 1+\log C \right)
+ O(1)\\
& \leq & 
\log \frac{d-i+1}{d} \left(\log d + 1+\log C \right) 
+ O(1)\\
& \leq & 
\log \left( \frac{1}{\e} - \frac{\alpha}{\log d} +O\left( \frac{1}{d} \right)
\right)
\left( \log d + 1+\log C \right) 
+ O(1)\\
& \leq & 
- \log d - \alpha \e  +O(1).
\end{eqnarray*}

For~$\alpha$ and $d$~large enough, we shall indeed have
$\|\vec{b}_j^*\| \leq \frac{\|\vec{b}_1\|}{\sqrt{d}}$ for any~$j \geq
i$.  Hence, since for this value of~$i$ we have
$\left(\frac{\sqrt{d-i+1}}{\sqrt{d}} \right)^{d-i+1} = 2^{-O(d)}$ and
$\left( \frac{\|\vec{b}_1\|}{\|\vec{b}_i^*\|}\right)^{d-i+1} =
d^{\frac{d}{2\e}} / 2^{O(d)}$, the lower bound
becomes~$d^{\frac{d}{2\e}} / 2^{O(d)}$, which concludes the
proof of the theorem.  \qed
\end{proof}

\subsection{On Schnorr's Constants}

First of all, we improve the best known upper bound for~$\alpha_k$
from~$k^{\log k+1}$ to~$k^{\frac{\log k}{2}+O(1)}$. We will see below
that this improved upper bound is essentially the best possible.

\begin{theorem}
\label{th:alpha_k}
Let~$k \geq 2$. Then~$\alpha_k \leq k^{\frac{\log k}{2} +O(1)}$.
\end{theorem}

\begin{proof}
Let~$(\vec{b}_1,\ldots,\vec{b}_k)$ be an HKZ-reduced basis. For any~$i$, we have
\[
\|\vec{b}_i^*\|^{k-i} \leq \sqrt{k-i+1}^{k-i+1} \prod_{j>i} \|\vec{b}_j^*\|
\]
Let the sequence~$u_i$ be defined by~$u_k = \|\vec{b}_k^*\|$
and~$u_i^{k-i} = \sqrt{k-i+1}^{k-i+1} \prod_{j>i} u_j$. Then the
sequence~$u_i$ dominates the sequence~$\|\vec{b}_i^*\|$. Moreover,
\[
\frac{u_i}{u_{i+1}} =
\frac{\sqrt{k-i+1}}{\sqrt{k-i}} \sqrt{k-i+1}^{\frac{1}{k-i}},
\]
which implies that
\[
\frac{\|\vec{b}_1\|}{\|\vec{b}_k^*\|} \leq \frac{u_1}{u_k} \leq 
\sqrt{k} \prod_{i < k} \sqrt{i}^{\frac{1}{i-1}}
\leq O(1) \sqrt{k} k^{\frac{\log k}{4}} .
\]
This concludes the proof. \qed
\end{proof}

We now show that the new upper bound on~$\alpha_k$ and the upper
bound~$\beta_k \leq \frac{1}{10}k^{2\log 2}$ are essentially the best
possible. They are in particular essentially reached for the
worst-case HKZ-reduced bases of the previous section.

\begin{theorem}
\label{th:Schnorr_constants}
Let~$k \geq 2$. We have:
\begin{eqnarray*}
\alpha_k = k^{\frac{\log k}{2} + O(1)} \ \mbox{ and } \ \beta_k = k^{2\log
  2+O\left(\frac{1}{\log k} \right)}.
\end{eqnarray*}
\end{theorem}

\begin{proof}
Consider a worst-case~$k$-dimensional HKZ-reduced basis as described
in the previous section. We have~$\|\vec{b}_k^*\|=1$, and  
$\|\vec{b}_1\| = k^{\log k - O(1)}$ follows from Corollary~\ref{co:shape}.

Now, we consider a worst-case~$2k$-dimensional HKZ-reduced
basis~$(\vec{b}_1,\ldots,\vec{b}_{2k})$ of a lattice~$L$ as described
in the previous section. We have the following lower bounds:
\[
\frac{\prod_{i \leq k} \|\vec{b}_i^*\|}{\prod_{i>k} \|\vec{b}_i^*\|} 
= \frac{\det(L)}{\prod_{i>k} \|\vec{b}_i^*\|^2} 
= 
\left(\frac{\sqrt{2k}}{\sqrt{k}}
\frac{\|\vec{b}_1\|}{\|\vec{b}_{k+1}^*\|}  \right)^{2k}.
\]
Furthermore, $\left(\frac{\|\vec{b}_1\|}{\|\vec{b}_{k+1}^*\|}\right)^4 =
\exp\left(\log^2(2k) - \log^2(k) + O(1)\right) = k^{2\log 2}
\exp(O(1))$, as claimed.  \qed
\end{proof}

\section{Difficult Bases for the BKZ Reductions}
\label{se:BKZ}
In this section, we build lattice bases that are~$k$-BKZ reduced, but far
from being fully HKZ-reduced. In the previous section, we showed lower
bounds to Schnorr's constants appearing in the quality analysis of the
hierarchies of reductions. Here we prove lower bounds on the quality
itself. Note that the lower bounds that we obtain are of the same
order of magnitude as the corresponding upper bounds, but the involved
constants are smaller. This suggests that it may not be possible to
combine worst cases for Schnorr's constants in order to build bad
bases for the BKZ hierarchy of reductions and that better upper bounds
may be proved by using an amortised analysis.

In the following, we fix a block-size~$k$. The strategy used to prove
the existence of the basis is almost the same as in
Section~\ref{se:ajtai}. The sole difference is that when we add a new
basis vector~$\vec{b}_j$, we only
require~$(\vec{b}_{j-k+1}(j-k+1),\ldots,\vec{b}_j(j-k+1))$ to be
HKZ-reduced instead of~$(\vec{b}_1, \ldots, \vec{b}_j)$. This
modification provides us the following result.

\begin{theorem}
\label{th:ajtai_for_BKZ}
Let~$d>k$ and $f: \iinterv{1, d} \rightarrow \mR^+ \setminus \{0\}$. 
Assume that for any~$j \leq d$, one has
\[
\sum_{i=\max(j-k+1,1)}^{j-1} 
\left(\frac{2\pi \mathrm{e}}{j-i}\right)^{\frac{j-i}{2}} 
\left(
1 - \left(\frac{f(j)}{f(i)}\right)^2\right)_+^{\frac{j-i}{2}} 
\left( \prod_{l=i}^j \frac{f(i)}{f(l)} \right) < 1.
\]
Then there exists a~$k$-BKZ-reduced basis
$(\vec{b}_1,\ldots,\vec{b}_d)$ with~$\|\vec{b}_i^*\| = f(i)$.
\end{theorem}

We now give a function~$f$ that fulfils the requirements of
Theorem~\ref{th:ajtai_for_BKZ}.

\begin{corollary}
\label{cor:bad_bkz0}
Let $k$ be an integer and $c<1$ be a constant such that
\[
\sum_{l=1}^{k-1} \left(\frac{4\pi \e}{lc} \sinh(-l \log
c)\right)^{\frac{l}{2}} < 1.
\]
Then, there exists a k-BKZ-reduced
basis~$(\vec{b}_1,\ldots,\vec{b}_d)$ with~$\|\vec{b}_i^*\| = c^i$.
\end{corollary}

\begin{proof}
Let~$f(i) = c^i$ for any~$i \leq d$. The condition of
Theorem~\ref{th:ajtai_for_BKZ} becomes
\[
\forall j \leq d, \ \sum_{i=\max(j-k+1,1)}^{j-1} 
\left(\frac{2\pi \e}{j-i}\left(1 - c^{2(j-i)}\right)c^{-(j-i+1)}
\right)^{\frac{j-i}{2}} < 1,
\]
or equivalently
\[
\forall j \leq d, \ \sum_{l=1}^{\min(k-1,j-1)}
\left(\frac{2\pi \e}{l}\left(1 - c^{2l}\right)c^{-(l+1)}\right)^{\frac{l}{2}} < 1.
\]
Since~$k < d$, this condition is equivalent to the one stated in 
the corollary. 
\qed
\end{proof}

Using the corollary above, one can compute a suitable constant~$c$ for
any given block-size. For~$k = 2$, one can take~$c = 0.972$, for~$k =
3$, one can take~$c = 0.985$ and for~$k \leq 10$, one can take~$c =
0.987$.  The optimal value of~$c$ seems to grow very slowly with~$k$.
However, it does grow since for any fixed~$c$, the general term of the
sum tends to~$+\infty$ when~$l$ grows to~$+\infty$.  We can also
derive the following general result, as soon as the block-size is
large enough:

\begin{corollary}
\label{cor:bad_bkz}
Let~$d>k > 8\pi \e$. There exists a~$k$-BKZ-reduced
basis~$(\vec{b}_1,\ldots,\vec{b}_d)$ of a lattice~$L$
with~$\|\vec{b}_i^*\| = \left(\frac{8\pi
  \e}{k-1}\right)^{\frac{i}{k}}$.  In particular, for any such basis,
we have:
\[
\frac{\|\vec{b}_1\|}{\lambda(L)} \geq \sqrt{d} \left(\frac{k-1}{8\pi
  \e} \right)^{\frac{d-1}{2k}}.
\]

\end{corollary}
\begin{proof}
Let~$c=\left(\frac{8\pi \e}{k-1}\right)^{\frac{1}{k}}$ and~$\phi: x
\mapsto \frac{1}{x} \sinh (x \log c)$. We have that
\[
\phi'(x)= - \frac{1}{x^2} \sinh (x \log c) + \frac{\log c}{x} \cosh (x
\log c) = \frac{\cosh(x \log c)}{x^2} ( - \tanh(x \log c) + x\log c).
\]
Since~$\tanh x \leq x$ for any~$x<0$, we have that the function~$\phi$
decreases when~$x<0$. As a consequence, we obtain that for any~$l <
k$,
\[
\frac{4\pi \e}{lc} \sinh (-l\log c)
\leq
\frac{2 \pi \e}{(k-1)} c^{-k} 
\leq 1/4.
\]

It follows that the condition of Theorem~\ref{th:ajtai_for_BKZ} is
satisfied. It now remains to give a lower bound to~$\|\vec{b}_1\| /
\lambda(L)$. We have~$\|\vec{b}_1\| = 
\left(\frac{8\pi
  \e}{k-1}\right)^{\frac{1}{k}}$ and Minkowski's theorem gives us that
\[
\lambda (L) \leq
\sqrt{d} \left( \prod_i\|\vec{b}_i^*\| \right)^{\frac{1}{d}}
= \sqrt{d} \left(\frac{8\pi \e}{k-1}\right)^{\frac{d+1}{2k}}.
\]
This directly provides the second claim of the theorem.
\qed
\end{proof}

By comparing to~$1$ the last term of the sum in
Corollary~\ref{cor:bad_bkz0}, one sees that the following must hold:
\[
(c^{-k} - c^{k+2}) \leq \frac{k-1}{2\pi\e}.
\]
This means that, apart from replacing~$8\pi \e$ by~$2\pi \e$ in
Corollary~\ref{cor:bad_bkz}, one cannot hope for a much better constant
by using our technique.

\section{Concluding Remarks}
\label{se:open}

We showed the existence of bases that are particularly bad from
diverse perspectives related to strong lattice reductions and strong
lattice reduction algorithms. A natural extension of our work would be
to show how to generate such bases efficiently, for example by showing
that the probabilities of obtaining bases of the desired properties
can be made extremely close to~$1$. Another difficulty related to this
goal will be to transfer the results from the continuous model, i.e.,
$\mR^n$, to a discrete space, e.g., $\mQ^n$ with a bound on denominators. 

Our results allow to claim that some algorithms/reductions are better
than others from the worst-case asymptotic complexity point of view.
This only gives a new insight on what should be done in practice. It
is well-known (see~\cite{NgSt06} about the LLL algorithm)
that low-dimensional lattices may behave quite differently from
predicted by the worst-case high-dimensional results.

\section*{Acknowledgements}
This work was initiated during the July~2007 seminar ``Explicit
methods in Number Theory'' at Mathematisches Forschungsinstitut
Oberwolfach.  The authors are grateful to the MFO for the great
working conditions provided on this occasion.  The authors would also
like to thank Jacques Martinet for the interest he showed for a
preliminary version of those results and for pointing~\cite{PeZw07}.
The second author thanks John Cannon and the University of Sydney for
having hosted him while some of the present work was completed.

\newcommand{\SortNoop}[1]{}

\section*{Proof of Theorem~\ref{th:main}}

This section is devoted to proving Theorem~\ref{th:main}. Since~$\exp(5) > 
2\pi \e (\sqrt{\e}+1)^2$, it suffices to prove the following result.

\begin{theorem}
\label{th:main2}
Let~$\psi(x) = C \cdot x$ with~$C=\exp(-6)$. Then for 
all~$1 \leq i < j \leq d$, we have
\[
(j - i + 1)^{-\frac{j - i}{2}} 
\left(1 - \left(\frac{f_{\psi,d}(j)}{f_{\psi,d}(i)}\right)^2 \right)_+^{\frac{j-i}{2}} 
\left( \prod_{k=i}^j \frac{f_{\psi,d}(i)}{f_{\psi,d}(k)} \right) 
\leq  
\exp \left(- \frac{5}{2} (j-i)\right),
\]
where~$f_{\psi,d}(d)=1$ and~$f_{\psi,d}(i)= \sqrt{\psi(d-i+1)} \cdot \left(
\prod_{k=i}^d f_{\psi,d}(k) \right)^{\frac{1}{d-i+1}}$. 
\end{theorem}

%% TODO : un graphe de la fonction f serait le bienvenu!
%% surcharge de alpha et de beta.
%% Parler de Rankin
%% Parler de dual-HKZ.
%% definir \asym

We shall work separately with the following two terms of the theorem:
\[
\left(1 - \left(\frac{f_{\psi,d}(j)}{f_{\psi,d}(i)}\right)^2 \right)_+^{\frac{j-i}{2}} 
\ \mbox{ and } \ 
\left( \prod_{k=i}^j \frac{f_{\psi,d}(i)}{f_{\psi,d}(k)} \right).
\]
We call these terms~$T_1$ and~$T_2$. Another notation that we use
is~$a = d-i+1$ and~$b=d-j+1$, which is natural since the 
function~$x \mapsto f(d-x+1)$ does not depend on~$d$. The domain
of valid pairs~$(a,b)$ is~$1 \leq b < a \leq d$.

Notice that if~$j=d$, then we can use the definition of~$f_{\psi,d}$,
and by bounding~$T_1$ by~$1$, we obtain the sufficient condition:
\[
\sqrt{d-i+1} \exp(-3 (d-i+1)) \leq \exp \left(- \frac{5}{2} (d-i)\right),
\]
which is valid. In the following, we will assume that~$j<d$.

Our proof is made of four main steps. The first step consists in
simplifying the expressions of the terms~$T_1$ and~$T_2$. In the
second step, we try to obtain the result without the first term, i.e.,
while bounding~$T_1$ by~$1$. We reach this goal for~$a \geq 158000$
along with~$b \leq a - \frac{1.65}{\log^3 a}$. In the third step, we
use~$T_2$ to obtain the result for~$a \geq 158000$ along with~$b \geq a
- \frac{1.65}{\log^3 a}$. Finally, we prove the result for~$1 \leq b < a
\leq 158000$ with an exhaustive check of the inequality to be
satisfied.

\subsection{Explicit Formulas}

The results of this subsection remain correct for any function~$\psi$.

\begin{lemma}
\label{le:explf}
The following holds for any~$k > i$:
\[
\frac{f_{\psi, d}(i)}{f_{\psi, d}(k)}
=
\sqrt{\frac{\psi(d - i + 1)}{\psi(d - k + 1)}} \cdot
\prod_{\ell = i+1}^k \psi(d - \ell + 2)^{\frac{1}{2(d - \ell + 1)}}. 
\]
\end{lemma}

\begin{proof}
We have
\[
f_{\psi, d}(i)^{d-i} = \psi(d - i + 1)^{\frac{d - i + 1}{2}} 
\cdot \prod_{k = i + 1}^d f_{\psi, d}(k)\]
and
\[f_{\psi, d}(i+1)^{d-i} = \psi(d - i)^{\frac{d - i}{2}} \cdot 
\prod_{k = i + 1}^d f_{\psi, d}(k).\]
By taking the quotient, we obtain 
\[
\frac{f_{\psi, d}(i)}{f_{\psi, d}(i+1)} = 
\sqrt{\frac{\psi(d - i + 1)}{\psi(d - i)}} \cdot 
\psi(d - i + 1)^{\frac{1}{2(d - i)}}.
\]
The lemma follows by induction. 
\qed
\end{proof}

The following lemma simplifies the expression of the term~$T_2$.

\begin{lemma}
\label{le:explicitpsi}
The following holds for any~$j>i$:
\[
\prod_{k=i+1}^j \frac{f_{\psi, d}(i)}{f_{\psi, d}(k)}
= 
\left( \prod_{l=i+1}^j \frac{\psi(d-i+1) 
\psi(d-l+2)}{\psi(d-l+1)(d - l + 2)^{\frac{d-j}{d-l+1}}}
\right)^{\frac{1}{2}}.
\]
\end{lemma}

\begin{proof}
We have
\[
\prod_{k=i+1}^j \frac{f_{\psi, d}(i)}{f_{\psi, d}(k)} =
\left(\prod_{k=i+1}^d \frac{f_{\psi, d}(i)}{f_{\psi, d}(k)} \right)
\cdot \left(\prod_{k=j+1}^d \frac{f_{\psi, d}(j)}{f_{\psi, d}(k)} \right)^{-1}
\cdot \left(\frac{f_{\psi, d}(i)}{f_{\psi, d}(j)}\right)^{j-d}.
\]
The first two terms can be made explicit by using the definition 
of~$f_{\psi, d}$, and the last one has been studied in 
Lemma~\ref{le:explf}. We get: 
\begin{eqnarray*}
\prod_{k=i+1}^j \frac{f_{\psi, d}(i)}{f_{\psi, d}(k)} & = &  
\frac{ \psi(d-i+1)^{\frac{d-i+1}{2}}}{\psi(d-j+1)^{\frac{d-j+1}{2}}} \cdot 
\left( \frac{\psi(d-i+1)}{\psi(d-j+1)} \right)^{\frac{j-d}{2}} \cdot
\left( \prod_{l=i+1}^j \psi(d-l+2)^{\frac{j-d}{2(d-l+1)}} \right) \\
& = & 
\frac{\psi(d-i+1)^{\frac{j-i+1}{2}}}{\psi(d-j+1)^{\frac{1}{2}}} \cdot 
\prod_{l=i+1}^j \psi(d - l + 2)^{\frac{j-d}{2(d-l+1)}} \\
& = & 
\left( \prod_{l=i+1}^j \frac{\psi(d-i+1) 
\psi(d-l+2)}{\psi(d-l+1)\psi(d - l + 2)^{\frac{d-j}{(d-l+1)}}}
\right)^{\frac{1}{2}},
\end{eqnarray*}
as claimed. 
\qed
\end{proof}

Note that by writing~$a = d - i + 1$ and~$b = d - j + 1$, the two
lemmas above give us:
\[
T_1 = \left(1 - \frac{\psi(b)}{\psi(a)}
\prod_{l=b}^{a-1} \psi (l+1)^{-\frac{1}{l}} \right)_+^{\frac{a-b}{2}} 
\ \mbox{ and } \ 
T_2 = \left( \prod_{l=b}^{a-1} \frac{\psi(a) \psi(l+1)}
{\psi(l)\psi(l+1)^{\frac{b-1}{l}}} \right)^{\frac{1}{2}}.
\]

\subsection{Temptative Proof of Theorem~\ref{th:main2} Without Using~$T_1$}

We consider the logarithm of~$(j - i + 1)^{-\frac{j - i}{2}} T_2$ and
try to show that it is smaller than~$-\frac{5}{2}(j-i)$. Thanks to
Lemma~\ref{le:explicitpsi}, this is equivalent to showing that:
\begin{equation}
\label{eq:ab}
- (a-b) \log (a-b+1)
+  
\sum_{l=b}^{a-1} \left( \log \psi(a) - \log \psi(l) 
+ \log \psi(l+1) \left(1 - \frac{b-1}{l} \right) \right)
\leq  - 5(a-b).
\end{equation}

We first try to simplify the summand.

\begin{lemma}
\label{le:increase}
Let~$b \geq 2$ be an integer. The function $x\in[b, a-1] \mapsto -\log
x + \log (x+1) \left(1 - \frac{b-1}{x}\right)$ is increasing for~$x
\geq b$ if~$b \geq 3$ and for~$x \geq 4$ if~$b=2$.
\end{lemma}

\begin{proof}
The derivative is $\frac{\log(x+1)(b-1)(x+1) - bx}{x^2(x+1)}$.  It
follows that the function under study is increasing as soon as~$\left(1 +
\frac{1}{x}\right) \log (x+1) \geq \frac{b}{b-1}$.  The result follows
from the facts that~$\frac{b}{b-1} \leq 2$, that~$\frac{5}{4} \log 5 >
2$ and that~$\frac{4}{3}\log 4 > \frac{3}{2}$.  
\qed
\end{proof}

By using Lemma~\ref{le:increase}, we obtain an upper bound to~$T_2$ if we
had taken~$\psi (x) = x$ instead of~$\psi(x) = C\cdot x$.

\begin{lemma}
\label{le:compensation}
%$\beta(a, b) := \left(1 - \frac{b}{a-b} \log \frac{a}{b} \right) \log C$. 
The following holds for~$a \geq 8$:
\begin{eqnarray*}
\sum_{x=b}^{a-1} && \left( \log a - \log x + \log (x+1)
\left(1 - \frac{b-1}{x}\right) \right) \\
&& \leq 
(a - b) \log (a - b + 1) + (a - b) 
\left(\log\frac{a^2}{(a - 1)(a - b + 1)} - \frac{b-1}{a-1} \log a \right)
\end{eqnarray*}
\end{lemma}

\begin{proof}
When $b \geq 3$, the result follows directly from
Lemma~\ref{le:increase}, by noticing that for all $x\in[b,a-1]$ we
have
\[
- \log x + \log (x+1) \left(1 - \frac{b-1}{x}\right)
\leq 
- \log (a-1) + \log (a) \frac{a-b}{a-1}.
\]
Suppose now that~$b=2$. It can be checked numerically that the
inequality holds for~$a=8$. Suppose now that~$a>8$. We have:
\begin{eqnarray*}
\sum_{x=b}^{a-1} \left( \log a - \log x + \log (x+1)
\left(1 - \frac{1}{x}\right) \right) 
&\leq & 6 \log 7 + 6 \left(\log \frac{64}{49}-\frac{1}{7}\log 8 \right)\\
& + & \sum_{x=8}^{a-1} \left(\log a -\log (a-1)+\log (a) \frac{a-b}{a-1}
\right)\\
%%%
& = &
\sum_{x=2}^{a-1} \left(\log a -\log (a-1)+\log (a) \frac{a-b}{a-1}\right),
\end{eqnarray*}
which gives the result.
\qed
\end{proof}

Notice that Lemma~\ref{le:compensation} implies that~$T_2$ with~$\psi(x)=x$ 
instead of~$C\cdot x$ already compensates the term~``$(a-b) \log (a-b+1)$'' of
Equation~(\ref{eq:ab}). Indeed, the 
function~$\theta: b \mapsto \log\frac{a^2}{(a - 1)(a - b + 1)} - 
\frac{b-1}{a-1} \log a$ is convex and
\[
\theta(2) = 2\log \frac{a}{a - 1} - \frac{\log a}{a - 1} 
\ \mbox{ and } \
\theta(a - 1) = \log \frac{a}{2(a - 1)} + \frac{\log a}{a - 1}.
\]
Both~$\theta(2)$ and~$\theta(a-1)$, and thus all~$\theta(x)$ 
for~$x\in [2, a-1]$,  are~$\leq 0$ for~$a \geq 8$.

We now consider the left hand-side of Equation~(\ref{eq:ab})
with~$\psi(x) = C \cdot x$.

\begin{lemma}
\label{le:alphabeta}
Let~$\alpha(a, b) = \log \frac{a}{(a - b)} - \frac{b-1}{a-1} \log a$
and~$\beta(a, b) = 1 - \frac{b}{a-b} \log \frac{a}{b}$.
For~$a \geq 8$, we have:
\begin{eqnarray*}
- (a-b) \log (a-b+1)
& + & 
\sum_{l=b}^{a-1} \left( \log \psi(a) - \log \psi(l) 
+ \log \psi(l+1) \left(1 - \frac{b-1}{l} \right) \right)\\
& \leq &
(a-b) \left(\alpha(a, b) + \beta(a,b) \log C \right).
\end{eqnarray*}
\end{lemma}

\begin{proof}
First of all, we have: 
\[
- (a-b) \log (a-b+1) +  
\sum_{l=b}^{a-1} \left( \log a - \log l 
+ \log (l+1) \left(1 - \frac{b-1}{l} \right) \right)
\leq 
\alpha(a, b).
\]
This follows from Lemma~\ref{le:compensation} and the fact
that~$(a-1)(a-b+1) \geq a(a-b)$.  We now consider the terms depending
on~$C$. Since~$\sum_{x=b+1}^a \frac{1}{x} \leq \log \frac{a}{b}$
and~$\log C <0$, we have:
\[
\sum_{l=b}^{a-1} \left( \log (C) \left(1 - \frac{b-1}{l} \right) \right)
\leq \log(C) \left( a-b - (b-1) \log\frac{a}{b} \right) \leq
\log(C) \beta(a,b),
\]
which gives the result.
\qed
\end{proof}

In the following, we study the function~$(a,b) \mapsto \alpha(a,
b) + \beta(a,b) \log C$. We would like to bound it by~$-5$, be
we will be able to do this only for a subset of all possible values of
the pair~$(a,b)$. 

\begin{lemma}
\label{le:taudecroitalpha}
Let~$0 < \kappa < 1$ be a real constant and suppose that~$a \geq 8$. The
function~$a \mapsto \alpha(a, \kappa a) + \beta(a, \kappa a) \log C$
decreases with respect to~$a$.
\end{lemma}

\begin{proof}
We have
\[
\alpha(a, \kappa a) + \beta(a, \kappa a) \log C= 
-\log (1 - \kappa) + \log C \left(1 + \frac{\kappa \log \kappa}{1 - \kappa}
\right) - \frac{(\kappa a - 1) \log a}{a - 1}.
\]

Hence, 
\[
\frac{\partial}{\partial a} \left(\alpha(a, \kappa a) + \log C
\beta(a, \kappa a)\right) = \frac{-\kappa a^2 + a \log a (\kappa - 1)
  + (\kappa + 1) a - 1}{a(a-1)^2}.
\]

For the numerator to be negative, it suffices 
that~$a \geq 1 + \frac{1}{\kappa}$ (then the term
in~$a^2$ is larger than the term in~$a$) or 
that~$a \geq \exp\left(\frac{\kappa + 1}{1 - \kappa}\right)$ 
(then the term in~$a\log a$ is larger than the term in~$a$). 
Since 
\[
\max_{\kappa \in [0, 1]} \min \left(1 + \frac{1}{\kappa}, 
\exp\left(\frac{\kappa + 1}{1 - \kappa}\right) \right) \leq 6,
\]
the result follows. 
\qed
\end{proof}

In the results above, we did not need~$C=\exp(-6)$. The only property
we used about~$C$ was~$\log C<0$.  In the sequel, we define~$\tau(a,
\kappa) = \alpha(a, \kappa a) - 6 \beta(a, \kappa a)$.  We are to
prove that~$\tau(a, \kappa) \leq -5$ as soon as~$\kappa$ is not very
close to~$1$.

\begin{lemma}
\label{le:taudecroitkappa}
For any~$a \geq 755$, the function $\kappa \mapsto \tau(a, \kappa)$
increases to a local maximum in~$\left[0, \frac{1}{2}\right]$, then
decreases to a local minimum in~$\left[\frac{1}{2}, 1 - \frac{1}{2
    \log a}\right]$ and then increases.
\end{lemma}

\begin{proof}
We first study 
\[
\frac{\partial^3}{\partial \kappa^3} \tau(a, \kappa) =
\frac{20\kappa^2+10\kappa^3+6-36\kappa-36 \kappa^2 \log
  \kappa}{(1-\kappa)^4 \kappa^2}.
\]

Using the fact that $\log \kappa \leq (\kappa - 1) - (\kappa - 1)^2/2
+(\kappa - 1)^3/3$ for $\kappa \in [0, 1]$, we find that the numerator
can be lower bounded by a polynomial which is non-negative for~$\kappa
\in [0, 1]$. As a consequence, $\tau'_{\kappa} (a, \kappa)=
\frac{\partial}{\partial \kappa} \tau(a, \kappa)$ is a convex function
with respect to~$\kappa\in (0,1)$. 

Notice now that~$\tau'_{\kappa} (a, \kappa) = -6\log \kappa + o(\log
\kappa) > 0$ for~$\kappa$ close to $0$, that~$\tau'_\kappa(a, 1/2) =
-10 + 24\log 2 - \frac{a\log a}{a-1} \leq 0$ for~$a \geq 755$, and
finally that
\begin{eqnarray*}
\tau'_\kappa\left(a, 1-\frac{1}{2\log a}\right) 
&=& 
-10 \log a -24 \log \left(1 - \frac{1}{2\log a}
\right) \log^2 a - \frac{a}{a-1} \log a\\
& \geq & 
2 \log a - \frac{a}{a-1} \log a,
\end{eqnarray*}
which is clearly positive for~$a \geq 3$. 
\qed
\end{proof}

The following lemma provides the result claimed in Theorem~\ref{th:main2}
for~$a \geq 158000$ and~$b \leq a - 1.65 \frac{a}{\log^3 a}$.

\begin{lemma}
\label{le:alphabeta}
Suppose that~$a \geq 158000$. Then, for 
all~$\kappa \leq 1 - 1.65\frac{1}{\log^3 a}$, we 
have~$\alpha(a, \kappa a) - 6 \beta(a, \kappa a) \leq -5$.
\end{lemma}

\begin{proof}
Let~$a_0= 158000$.  We have~$\tau'_\kappa(a_0, 0.08962) > 0 >
\tau'_\kappa(a_0, 0.08963)$.  Furthermore, for $\kappa \in [0.0937,
  0.0938]$, we have
\[
\left| \tau'_\kappa(a_0, \kappa) \right| \leq 
\max \left( |\tau'_\kappa(a_0, 0.08962)|, |\tau'_\kappa(a_0, 0.08963)|\right) 
\leq 3 \cdot 10^{-4}.
\]
Hence, 
\[
\max_{\kappa\in[0.08962, 0.08963]} \tau(a_0, \kappa) \leq 
\tau(a_0, 0.08962) + 3\cdot 10^{-9} < -5.
\]

Thanks to Lemmas~\ref{le:taudecroitalpha} and \ref{le:taudecroitkappa}, 
we have, for $a \geq 158000$: 
\[
\max_{\kappa\in [0, 1/2]}
\left( \alpha(a, \kappa a) - 6 \beta(a, \kappa a) \right) \leq -5.
\]

Furthermore, since~$\frac{1}{2\log a} \geq \frac{1.65}{\log(a)^3}$ and thanks
to Lemma~\ref{le:taudecroitkappa}, we have, for any~$a \geq 158000$:
\[
\max_{\kappa\in\left[\frac{1}{2}, 1-\frac{1.65}{\log^3 a} \right]} \tau(a, \kappa)
= \max\left( \tau\left(a, \frac{1}{2} \right), 
\tau \left(a,  1-\frac{1.65}{\log^3 a}\right) \right).
\]
Notice that 
\[
\tau\left(a,  1-\frac{1.65}{\log^3 a} \right)
\leq \alpha \left(a, a - \frac{1.65a}{\log^3 a} \right)
= 
-\log 1.65 + 3 \log \log a - \log a + \frac{a}{a-1} \frac{1.65}{(\log a)^2},
\]
which is decreasing with respect to~$a \geq 158000$. Moreover,
for~$a=158000$, its value is below~$-5$. As a consequence,
\[
\max_{\kappa\in \left[ \frac{1}{2}, 1-\frac{1.65}{\log^3 a}\right]} 
\tau(a, \kappa) \leq \max\left(\tau\left(a, \frac{1}{2}\right), -5\right) 
\leq -5.
\]
\qed\end{proof}

\subsection{Using~$T_1$ When~$b > a- \frac{1.65 a}{(\log a)^3}$}

This section ends the proof of Theorem~\ref{th:main2} for~$a \geq 158000$.

\begin{lemma}
\label{le:abveryclose}
Assume that~$\psi(x) = e^{-6} \cdot x$. Then, for~$a > b \geq a - 1.65
\frac{a}{(\log a)^3}$ and~$a \geq a_1 \geq 1782$, we have
\[
1 - \left(\frac{f_{\psi, d}(d-b+1)}{f_{\psi, d}(d-a+1)}\right)^2 \leq 
1 - \exp\left(-1.65\frac{\log a_1 - 5}{\log^3 a_1 - 1.65}\right).
\]
\end{lemma}

\begin{proof}
According to Lemma~\ref{le:explf}, we have 
\begin{eqnarray*}
-2 \log \frac{f_{\psi, d}(d-b+1)}{f_{\psi, d}(d-a+1)} & = & 
\log \left( \frac{a}{b} \right) + 
\sum_{l=b}^{a-1} \frac{-6 + \log (l+1)}{l}\\
& \leq & 
\frac{1.65}{\log^3 a - 1.65} + 
(a - b) \frac{-6 + \log a}{b},\\
& \leq &  
1.65 \frac{\log a - 5}{(\log a)^3 - 1.65}.\\
\end{eqnarray*}
This upper bound decreases with respect to~$a \geq 1782$.
\qed\end{proof}

By using Lemma~\ref{le:alphabeta} and the fact that~$\beta(a,b) \leq
0$, we see that the left hand side of Equation~(\ref{eq:ab}) is upper
bounded, for~$b \geq a - 1.65 \frac{a}{(\log a)^3}$ and~$a \geq a_1
\geq 1782$, by:
\[
(a-b) \log \left( 1- \exp\left( -1.65 \frac{\log a_1 -5}{\log^3 a_1 -1.65}
\right) \right) 
\leq 
(a-b) \log \left(1.65 \frac{\log a_1 -5}{\log^3 a_1 -1.65}
\right),
\]
and the constant in the right hand side is below $-5$ when $a_1 = 158000$.

\subsection{Small Values of~$a$}

It only remains to prove Theorem~\ref{th:main2} for small values
of~$a$.  The following lemma was obtained numerically. In order to
provide a reliable proof, we used the Boost interval arithmetic
library~\cite{BrMePi06} and CRlibm~\cite{crlibmweb} as underlying
floating-point libraries.

\begin{lemma}
\label{le:numeric}
Let ~$\psi(x) = e^{-6} \cdot x$. For any~$2 \leq b < a \leq 158000$, we have
\[
(j-i+1)^{-\frac{j-i}{2}} \left( 1 - \left(\frac{f_{\psi, d}(j)}{f_{\psi,
    d}(i)}\right)^2\right)_+^{\frac{j-i}{2}} \cdot \prod_{k=i+1}^j \frac{f_{\psi,
    d}(i)}{f_{\psi, d}(k)} \leq \exp \left(-5 \frac{j-i}{2} \right),\]
with~$i=d-a+1$ and~$j=d-b+1$.
\end{lemma}

\subsection{Concluding Remarks}
The value of~$C = \exp(-6)$ is not optimal. Given the line of proof used
above (obtaining a geometric decreasing of the general term of the sum 
in Theorem~\ref{th:ajtai}), the best value of~$C$ that one can expect is 
limited by the term corresponding to~$j = d$, $i = d-1$, for which
we must have~$(2 \pi \e) \cdot (2C) \leq \frac{1}{(\sqrt{\e}+1)^2}.$

Note however that the probability $p$ of Lemma~\ref{le:core} involved
in our criterion can be computed more precisely for small dimensional 
lattices, thus improving the optimal value of~$C$ that can be reached. 
%Combining this and a slightly less conservative line of proof, one 
%could hope that it is possible to take~$C$ as close as desired 
%to~$2\pi \e$. 

\newcommand{\SortNoop}[1]{}
\begin{thebibliography}{10}

\bibitem{Ajtai03}
M.~Ajtai.
\newblock The worst-case behavior of {S}chnorr's algorithm approximating the
  shortest nonzero vector in a lattice.
\newblock In {\em Proceedings of the 35th Symposium on the Theory of Computing
  ({STOC 2003})}, pages 396--406. ACM Press, 2003.

\bibitem{AjKuSi01}
M.~Ajtai, R.~Kumar, and D.~Sivakumar.
\newblock A sieve algorithm for the shortest lattice vector problem.
\newblock In {\em Proceedings of the 33rd Symposium on the Theory of Computing
  ({STOC 2001})}, pages 601--610. ACM Press, 2001.

\bibitem{BrMePi06}
H.~Br{\"o}nnimann, G.~Melquiond, and S.~Pion.
\newblock The design of the {B}oost interval arithmetic library.
\newblock {\em Theoretical Computer Science}, 351:111--118, 2006.

\bibitem{Cassels71}
J.~W.~S. Cassels.
\newblock {\em An Introduction to the Geometry of Numbers, 2nd edition}.
\newblock Springer-Verlag, 1971.

\bibitem{Cohen95}
H.~Cohen.
\newblock {\em A Course in Computational Algebraic Number Theory, 2nd edition}.
\newblock Springer-Verlag, 1995.

\bibitem{crlibmweb}
{CRLibm}, a library of correctly rounded elementary functions in
  double-precision.
\newblock \url{http://lipforge.ens-lyon.fr/www/crlibm/}.

\bibitem{GaHoKoNg06}
N.~Gama, N.~Howgrave-Graham, H.~Koy, and P.~Nguyen.
\newblock Rankin's constant and blockwise lattice reduction.
\newblock In {\em Proceedings of {C}rypto 2006}, number 4117 in Lecture Notes
  in Computer Science, pages 112--130. Springer-Verlag, 2006.

\bibitem{HaSt07}
G.~Hanrot and D.~Stehl{\'e}.
\newblock Improved analysis of {K}annan's shortest lattice vector algorithm
  (extended abstract).
\newblock In {\em Proceedings of {C}rypto 2007}, volume 4622 of {\em Lecture
  Notes in Computer Science}, pages 170--186. Springer-Verlag, 2007.

\bibitem{Helfrich85}
B.~Helfrich.
\newblock Algorithms to construct {M}inkowski reduced and {H}ermite reduced
  lattice bases.
\newblock {\em Theoretical Computer Science}, 41:125--139, 1985.

\bibitem{KaLe78}
A.~Kabatyanskii and V.~I. Levenshtein.
\newblock Bounds for packings. on a sphere and in space.
\newblock {\em Proulcmy Peredacha informats\"u}, 14:1--17, 1978.

\bibitem{Kannan83}
R.~Kannan.
\newblock Improved algorithms for integer programming and related lattice
  problems.
\newblock In {\em Proceedings of the 15th Symposium on the Theory of Computing
  ({STOC 1983})}, pages 99--108. ACM Press, 1983.

\bibitem{LeLeLo82}
A.~K. Lenstra, H.~W. Lenstra, Jr., and L.~Lov{\'a}sz.
\newblock Factoring polynomials with rational coefficients.
\newblock {\em Mathematische Annalen}, 261:513--534, 1982.

\bibitem{MiGo02}
D.~Micciancio and S.~Goldwasser.
\newblock {\em Complexity of lattice problems \hspace*{-1mm}: a cryptographic
  perspective}.
\newblock Kluwer Academic Press, 2002.

\bibitem{NgSt06}
P.~Nguyen and D.~Stehl{\'e}.
\newblock {LLL} on the average.
\newblock In {\em Proceedings of the 7th {A}lgorithmic {N}umber {T}heory
  {S}ymposium ({ANTS VII})}, volume 4076 of {\em Lecture Notes in Computer
  Science}, pages 238--256. Springer-Verlag, 2006.

\bibitem{NgSt01}
P.~Nguyen and J.~Stern.
\newblock The two faces of lattices in cryptology.
\newblock In {\em Proceedings of the 2001 Cryptography and Lattices Conference
  ({CALC'01})}, volume 2146 of {\em Lecture Notes in Computer Science}, pages
  146--180. Springer-Verlag, 2001.

\bibitem{NgVi07}
P.~Nguyen and T.~Vidick.
\newblock Sieve algorithms for the shortest vector problem are practical.
\newblock Submitted.

\bibitem{PeZw07}
R.~A. Pendavingh and S.~H.~M. van Zwam.
\newblock New {K}orkin-{Z}olotarev inequalities.
\newblock {\em SIAM Journal on Optimization}, 18(1):364--378, 2007.

\bibitem{Schnorr07}
C.~P. Schnorr.
\newblock Progress on {LLL} and lattice reduction.
\newblock In {\em Proceedings of the {LLL+25} conference}. To appear.

\bibitem{Schnorr87}
C.~P. Schnorr.
\newblock A hierarchy of polynomial lattice basis reduction algorithms.
\newblock {\em Theoretical Computer Science}, 53:201--224, 1987.

\bibitem{Schnorr94}
C.~P. Schnorr.
\newblock Block reduced lattice bases and successive minima.
\newblock {\em Combinatorics, Probability and Computing}, 3:507--533, 1994.

\end{thebibliography}
\end{document}